\documentclass{birkjour}
\usepackage{cancel,graphicx}
\usepackage{amsmath,amssymb}
\usepackage[english]{babel}

 \newtheorem{theorem}{Theorem}[section]

 \newtheorem{proposition}[theorem]{Proposition}
 \theoremstyle{definition}
 \newtheorem{definition}[theorem]{Definition}
 \theoremstyle{remark}
 \newtheorem{remark}[theorem]{Remark}
 
 \numberwithin{equation}{section}

\newcommand\caF{{\mathcal F}}

\newcommand\caP{{\mathcal P}}

\newcommand\caM{{\mathcal M}}
\newcommand\caS{{\mathcal S}}
\newcommand\caO{{\mathcal O}}
\newcommand\caU{{\mathcal U}}

\newcommand\gR{{\mathbb R}}

\newcommand\gM{{\mathbb M}}

\newcommand\gB{{\mathbb B}}
\newcommand\gS{{\mathbb S}}
\newcommand\gN{{\mathbb N}}

\newcommand\algA{{\mathbf A}}

\newcommand\kg{{\mathfrak g}}

\newcommand\kS{{\mathfrak S}}

\newcommand\Ad{{\text{\textup{Ad}}}}

\newcommand\fois{\mathord{\cdot}}

\newcommand\dd{{\text{\textup{d}}}}

\newcommand\norm{\mathord{\parallel}}
\newcommand\defin{\bf}

\begin{document}

%
%
%
%
%
%
%
%
%

\title[Star-exponential of the Poincar\'e Group]
 {Non-formal deformation quantization and star-exponential of the Poincar\'e Group}

\author[P. Bieliavsky]{Pierre Bieliavsky}

\address{%
IRMP, Universit\'e Catholique de Louvain}

\email{Pierre.Bieliavsky@uclouvain.be}

\author[A. de Goursac]{Axel de Goursac}

\address{%
Charg\'e de Recherche au F.R.S.-FNRS\\
IRMP, Universit\'e Catholique de Louvain}

\email{Axelmg@melix.net}

\author[F. Spinnler]{Florian Spinnler}

\address{%
IRMP, Universit\'e Catholique de Louvain,\\
Chemin du Cyclotron, 2\\
B-1348 Louvain-la-Neuve, Belgium}

\email{Florian.Spinnler@uclouvain.be}




\date{October 1, 2012}
\dedicatory{To the memory of Boris Fedosov}

\begin{abstract}
We recall the construction of non-formal deformation quantization of the Poincar\'e Group $ISO(1,1)$ on its coadjoint orbit and exhibit the associated non-formal star-exponentials.
\end{abstract}

\maketitle

\section{Introduction}
Quite generally, a theory -either physical or mathematical- often consists in the kinematical data of an algebra $\algA$  together with some dynamical data encoded 
by some specific, say ``action", functional $\kS$ on $\algA$. In most cases, the algebra is associative because it represents the ``observables"
in the theory, i.e. operators whose spectrality is associated with measurements. Also, the understanding of the theory 
passes through the determination of critical points of the action functional $\kS$, or sometimes equivalently, by the determination
of first integrals: elements of $\algA$ that will be preserved by the dynamics. Disposing of a sufficient numbers of such
first integrals yields a satisfactory understanding of the system. In particular, the consideration of symmetries appears as a necessity 
- rather than a simplifying hypothesis.

Now, assume that one disposes of enough such symmetries in order to entirely determine the dynamics. In that case, the 
system $(\algA,\kS)$ could be called ``integrable" or even, quite abusively, ``free". Of course, once a ``free" system is understood, one wants to pass to a perturbation or ``singularization" of it, for instance by implementing some type of ``interactions".  This is rather clear within 
the physical context. Within the mathematical context, such a singularization could for example correspond to implementing
a foliation.
However, once perturbed, the problem remains the same: determining symmetries. One may naively hope that 
the symmetries of the unperturbed ``free" system would remain symmetries of its perturbation. It is not the case: perturbing
generally implies symmetry breaking.
\medskip

One idea, due essentially to Drinfel'd, is to define the perturbation process through the data of the symmetries themselves \cite{Drinfeld:1989} in the framework of deformation quantization \cite{Bayen:1978,Fedosov:1994}. 
This allows, even in the case of a symmetry breaking, to control the ``perturbed symmetries".
More specifically, let $G$ be a Lie group with Lie algebra $\kg$ whose enveloping (Hopf) algebra is denoted by $\caU(\kg)$.
Consider the category of $\caU(\kg)$-module algebras i.e. associative algebras that admit an (infinitesimal) action of $G$.
A (formal) {\defin Drinfel'd twist} based on $\caU(\kg)$ is an element $F$ of the space of formal power series $\caU(\kg)\otimes\caU(\kg)[[\nu]]$ of the form
$F\;=\;I\otimes I+...$ satisfying a specific cocycle property (see e.g. \cite{Giaquinto:1998}) that ensures that for every $\caU(\kg)$-module algebra $(\algA,\mu_\algA)$ the formula $\mu_{\algA}^F\;:=\;\mu_\algA\circ F$ defines an associative algebra structure $\algA_F$ on the space $\algA[[\nu]]$. Disposing 
of a Drinfel'd twist then allows to deform the above mentioned category. However, as expected, the deformed objects are no longer 
$\caU(\kg)$-module algebras. But, the data of the twist allows to define a Hopf deformation of the enveloping algebra: keeping the 
multiplication unchanged, one deforms the co-product $\Delta$ of $\caU(\kg)$ by conjugating under $F$. This yields a new
co-multiplication $\Delta_F$ that together with the undeformed multiplication underly a structure of Hopf algebra $\caU(\kg)_F$ on $\caU(\kg)[[\nu]]$. The latter so called \emph{non-standard quantum group} $\caU(\kg)_F$ now acts on every deformed algebra $\algA_F$.
\medskip

At the {\defin non-formal} level, the notion of Drinfel'd twist based on $\caU(\kg)$ corresponds to the one of {\defin universal deformation formula} for the actions $\alpha$ of $G$ on associative algebras $\algA$ of a specified topological type such as Fr\'echet- or C*-algebras. This -roughly- consists in the data of a two-point kernel $K_\theta\in C^\infty(G\times G)$ ($\theta\in\gR_0$) satisfying specific properties that guarantee 
a meaning to integral expressions of the form $a\star^\algA_\theta b\;:=\;\int_{G\times G}K_\theta(x,y)\mu_\algA(\alpha_xa\otimes\alpha_yb)\,{\rm d}x{\rm d}y$ with $a,b\in\algA$. Once well-defined, one also requires associativity of the product $\star^\algA_\theta$ as well as 
the semi-classical limit condition: $\lim_{\theta\to0}\star_\theta^\algA=\mu_\algA$ in some precise topological context. This has been performed for abelian Lie groups in \cite{Rieffel:1993} and for abelian supergroups in \cite{Bieliavsky:2010su,deGoursac:2011kv}.

In \cite{Bieliavsky:2010kg}, such universal deformation formulae have been constructed for every Piatetskii-Shapiro normal $J$-group $\gB$.
For example, the class of normal $J$-groups (strictly) contains all Iwasawa factors of Hermitean type non-compact simple Lie groups.
A universal deformation formula in particular yields a left-invariant associative function algebra on the group $\gB$. It is therefore natural to 
ask for a comparison with the usual group convolution algebra. Following ideas mainly due to Fronsdal in the context of the $\star$-representation program (representation theory of Lie group in the framework of formal $\star$-products), one may expect 
that (a non-formal version of) the notion of {\defin star-exponential} \cite{Bayen:1978,Bayen:1982} plays a crucial role in this comparison. Moreover, such a star-exponential can give access to the spectrum of operators \cite{Cahen:1984,Cahen:1985} determining possible measurements of a system.
\medskip

In this paper, we recall the construction of the non-formal deformation quantization (see \cite{Bieliavsky:2008or}) and exhibit its star-exponential for the basic case of the Poincar\'e Group $ISO(1,1)$. Such a low-dimensional case illustrate the general method developed for normal $J$-group $\gB$ (which are Kahlerian), in \cite{Bieliavsky:2010kg} for star-products and in \cite{Bieliavsky:2013sk} for star-exponentials. However, the Poincar\'e group is solvable but of course not Kahlerian, so this paper shows also that the method introduced in \cite{Bieliavsky:2010kg,Bieliavsky:2013sk} can be extended to some solvable but non-Kahlerian Lie groups.

\section{Geometry of the Poincar\'e Group}

We recall here some features concerning the geometry of the Poincar\'e group $G=ISO(1,1)=SO(1,1)\ltimes\gR^2$ and of its coadjoint orbits. First, it is diffeomorphic to $\gR^3$, so let us choose a global coordinate system $\{(a,\ell,m)\}$ of it. Its group law can be read as
\begin{equation*}
(a,\ell,m)\fois(a',\ell',m')=(a+a',e^{-2a'}\ell+\ell',e^{2a'}m+m')
\end{equation*}
Its neutral element is $(0,0,0)$ and the inverse is given by: $(a,\ell,m)^{-1}=(-a,-e^{2a}\ell,-e^{-2a}m)$. By writing $(a,\ell,m)=\exp(aH)\exp(\ell E)\exp(mF)$, we can determine its Lie algebra $\kg$:
\begin{equation*}
[H,E]=2E,\qquad [H,F]=-2F,\qquad [E,F]=0.
\end{equation*}

Let us have a look to the coadjoint orbit of $G$. After a short calculation, one can find that
\begin{equation*}
\Ad^*_g(\alpha H^*+\beta E^*+\gamma F^*)= (\alpha+2\beta\ell-2\gamma m)H^*+\beta e^{-2a}E^*+\gamma e^{2a}F^*
\end{equation*}
if $g=(a,\ell,m)\in G$ and $\{H^*,E^*,F^*\}$ is the basis of $\kg^\ast$ dual of $\{H,E,F\}$. A generic orbit of $G$ is therefore a hyperbolic cylinder. We will study in particular the orbits associated to the forms $k(E^\ast-F^\ast)$ with $k\in\gR^\ast_+$, which will be denoted by $\gM_k$ or simply by $\gM$. The Poincar\'e quotient $\gM_k$ is globally diffeomorphic to $\gR^2$, so we choose the following coordinate system:
\begin{equation}
(a,\ell):=\Ad^*_{(a,\ell,0)}k(E^\ast-F^\ast)=k(2\ell H^*+e^{-2a}E^*-e^{2a}F^*).\label{eq-ident}
\end{equation}
$\gM$ is a $G$-homogeneous space for the coajoint action:
\begin{multline}
(a,\ell,m)\fois(a',\ell'):=\Ad^*_{(a,\ell,m)}k(2\ell' H^*+e^{-2a'}E^*-e^{2a'}F^*)\\
=(a+a',\ell'+e^{-2a'}\ell+e^{2a'}m).\label{eq-act}
\end{multline}

\begin{remark}
Note that the affine group $\gS$ (connected component of the identity of ``ax+b'') is the subgroup of $G$ generated by $H$ and $E$, i.e. by simply considering the two first coordinates $(a,\ell)$ of $G$. Actually the identification \eqref{eq-ident} yields a diffeomorphism between $\gS$ and $\gM$ which is $\gS$-equivariant with respect to the left action of $\gS$ and its action on $\gM$ by \eqref{eq-act} as a subgroup of $G$. This identification $\gS\simeq\gM$ is useful to construct star-products.
\end{remark}

The fundamental fields of the action \eqref{eq-act}, defined by
\begin{equation*}
X^\ast_{(a,\ell)}f=\frac{\dd}{\dd t}|_{t=0}f(\exp(-tX)\fois(a,\ell)),
\end{equation*}
for $X\in\kg$, $(a,\ell)\in\gM$, $f\in C^\infty(\gM)$,
are given by
\begin{equation*}
H^*_{(a,\ell)}=-\partial_a,\qquad E^*_{(a,\ell)}=-e^{-2a}\partial_\ell,\qquad F^*_{(a,\ell)}=-e^{2a}\partial_\ell.
\end{equation*}
This permits to compute the {\defin Kostant-Kirillov-Souriau symplectic form} of $\gM$,  $\omega_\varphi(X^*_\varphi,Y^*_\varphi):=\langle\varphi,[X,Y]\rangle$, for $\varphi\in\gM\subset\kg^\ast$, and $X,Y\in\kg$. One finds
\begin{equation}
\omega_{(a,\ell)}=2k\dd a\wedge\dd\ell. \label{eq-kks}
\end{equation}
For different values of $k\in\gR^\ast_+$, the $(\gM_k,\omega)$ are symplectomorphic, so we set $k=1$ in the following. Since the action of $G$ on its coadjoint orbit $\gM$ is strongly hamiltonian, there exists a Lie algebra homomorphism $\lambda:\kg\to C^\infty(\gM)$ (for the Poisson bracket on $\gM$ associated to $\omega$), called the {\defin moment map} and given by
\begin{equation}
\lambda_H=2\ell,\qquad \lambda_E=e^{-2a},\qquad \lambda_F=-e^{2a}.\label{eq-moment}
\end{equation}

\begin{proposition}
The exponential of the group $G$ is given by
\begin{equation*}
e^{tX}=\Big(\alpha t,\frac{\beta}{\alpha}e^{-\alpha t}\sinh(\alpha t),\frac{\gamma}{\alpha}e^{\alpha t}\sinh(\alpha t)\Big)
\end{equation*}
for $X=\alpha H+\beta E+\gamma F$.
\end{proposition}
\begin{proof}
It is a direct calculation using the semigroup property $e^{(t+s)X}=e^{sX}e^{tX}$ and by deriving by $s$.
\end{proof}

For $g=(a,\ell,m)\in G$, we can obtain straighforwardly the logarithm by inversing the above equation:
\begin{equation}
\log(a,\ell,m)= aH+\frac{ae^a\ell}{\sinh(a)}E+\frac{ae^{-a}m}{\sinh(a)}F\label{eq-log}
\end{equation}
and the {\defin BCH expression}:
\begin{multline}
\text{BCH}(X_1,X_2)=\log(e^{X_1}e^{X_2})=(\alpha_1+\alpha_2)H\\
+\frac{(\alpha_1+\alpha_2)}{\sinh(\alpha_1+\alpha_2)}\Big(\frac{\beta_1}{\alpha_1}e^{-\alpha_2}\sinh(\alpha_1)+\frac{\beta_2}{\alpha_2}e^{\alpha_1}\sinh(\alpha_2)\Big)E\\
+\frac{(\alpha_1+\alpha_2)}{\sinh(\alpha_1+\alpha_2)}\Big(\frac{\gamma_1}{\alpha_1}e^{\alpha_2}\sinh(\alpha_1)+\frac{\gamma_2}{\alpha_2}e^{-\alpha_1}\sinh(\alpha_2)\Big)F\label{eq-bchGroup}
\end{multline}
for $X_i=\alpha_i H+\beta_i E+\gamma_i F\in\kg$.

\section{Deformation quantization}

\subsection{Star-products}

Due to the identification $\gM\simeq\gR^2$, we can endow the space of Schwartz functions $\caS(\gR^2)$ with the Moyal product associated to the constant KKS symplectic form \eqref{eq-kks}:
\begin{equation}
(f\star^0_\theta h)(a,\ell)=\frac{4}{(\pi\theta)^2}\int\dd a_i\dd \ell_i\ f(a_1+a,\ell_1+\ell)
h(a_2+a,\ell_2+\ell) e^{-\frac{4i}{\theta}(a_1\ell_2-a_2\ell_1)}\label{eq-moyal}
\end{equation}
for $f,h\in\caS(\gR^2)$. It turns out that this associative star-product is {\defin covariant} for the moment map \eqref{eq-moment}, formally in the deformation parameter $\theta$:
\begin{equation}
\forall X,Y\in\kg\quad:\quad [\lambda_X,\lambda_Y]_{\star_\theta^0}=-i\theta\lambda_{[X,Y]}.\label{eq-cov}
\end{equation}
Nonetheless, it is not $G$-{\defin invariant}, one does not have:
\begin{equation}
\forall g\in G\quad:\quad g^\ast(f\star_\theta^0 h)= (g^\ast f)\star_\theta^0 (g^\ast h)\label{eq-invar}
\end{equation}
in general, where $g^\ast$ means the pullback of the action \eqref{eq-act} of $G$ on $\gM$: $g^\ast f:=f(g\,\fois)$. In the following, we exhibit intertwining operators $T_\theta$ (see \cite{Bieliavsky:2002}) in order to construct invariant star-products on $\gM$, i.e. satisfying \eqref{eq-invar}.

We consider $\caP_\theta$ a invertible multiplier on $\gR$: $\caP_\theta\in\caO_M^{\times}(\gR)=\{f\in C^\infty(\gR),\ \forall h\in\caS(\gR)\ f.h\in\caS(\gR)\text{ and } f^{-1}.h\in\caS(\gR)\}$, and $\phi_\theta$ defined by:
\begin{equation*}
\phi_\theta(a,\ell)=\Big(a,\frac{2}{\theta}\sinh(\frac{\theta\ell}{2})\Big),\qquad \phi_\theta^{-1}(a,\ell)=\Big(a,\frac{2}{\theta}\text{arcsinh}(\frac{\theta\ell}{2})\Big).
\end{equation*}
We define the operator $T_\theta=\caP_\theta(0)\caF^{-1}\circ(\phi_\theta^{-1})^\ast\circ \caP_\theta^{-1}\circ\caF$, from $\caS(\gR^2)$ to $\caS'(\gR^2)$, where $\caP_\theta^{-1}$ acts by multiplication by $\caP_\theta(\ell)^{-1}$ and the partial Fourier transformation is given by:
\begin{equation}
\caF f(a,\xi)=\hat f(a,\xi):=\int\dd\ell\ e^{-i\xi\ell}f(a,\ell).\label{eq-fourier}
\end{equation}
The normalization is chosen so that $T_\theta 1=1$. On its image, $T_\theta$ is invertible. The explicit expressions are:
\begin{align*}
&T_\theta f(a,\ell)=\frac{\caP_\theta(0)}{2\pi}\int\dd t\dd\xi\ \cosh(\frac{\theta t}{2})\caP_\theta(t)^{-1}e^{\frac{2i}{\theta}\sinh(\frac{\theta t}{2})\ell-i\xi t}f(a,\xi)\\
&T^{-1}_\theta f(a,\ell)=\frac{1}{2\pi\caP_\theta(0)}\int\dd t\dd\xi\ \caP_\theta(t)e^{-\frac{2i}{\theta}\sinh(\frac{\theta t}{2})\xi+it\ell}f(a,\xi)
\end{align*}
In \cite{Bieliavsky:2002}, it has been shown that this intertwining operator yields an associative product on $T_\theta(\caS(\gR^2))$: $f\star_{\theta,\caP}h:=T_\theta((T_\theta^{-1}f)\star_\theta^0(T_\theta^{-1}h))$ which is $G$-invariant. Its explicit expression is: $\forall f,h\in T_\theta(\caS(\gR^2))$,
\begin{multline}
(f\star_{\theta,\caP}h)(a,\ell)=\frac{4}{(\pi\theta)^{2}}\int\dd a_i\dd \ell_i \cosh(2(a_1-a_2))\frac{\caP_\theta(\frac{4}{\theta}(a_1-a))\caP_\theta(\frac{4}{\theta}(a-a_2))}{\caP_\theta(\frac{4}{\theta}(a_1-a_2))\caP_\theta(0)}\\
e^{\frac{2i}{\theta}(\sinh(2(a_1-a_2))\ell+\sinh(2(a_2-a))\ell_1+\sinh(2(a-a_1))\ell_2)}f(a_1,\ell_1)h(a_2,\ell_2).\label{eq-prod}
\end{multline}

\subsection{Schwartz space}

In \cite{Bieliavsky:2010kg}, a Schwartz space adapted to $\gM$ has been introduced, which is different from the usual one $\caS(\gR^2)$ in the global chart $\{(a,\ell)\}$ \eqref{eq-ident}.
\begin{definition}
\label{def-schwartz}
The {\defin Schwartz space} of $\gM$ is defined as
\begin{multline*}
\caS(\gM)=\{f\in C^\infty(\gM)\quad \forall \alpha=(k,p,q,n)\in\gN^4,\\
\norm f\norm_{\alpha}:=\sup_{(a,\ell)}\Big|\frac{\sinh(2a)^k}{\cosh(2a)^p}\ell^q\partial_a^p\partial_\ell^n f(a,\ell)\Big|<\infty\}.
\end{multline*}
\end{definition}
\noindent The space $\caS(\gM)$ corresponds to the usual Schwartz space in the coordinates $(r,\ell)$ with $r=\sinh(2a)$. It is stable by the action of $G$:
\begin{equation*}
\forall f\in\caS(\gM),\ \forall g\in G\quad:\quad g^\ast f\in\caS(\gM)
\end{equation*}
due to the formulation of the action of $G$ in the coordinates $(r,\ell)$:
\begin{multline*}
(r,\ell,m)(r',\ell')=\Big(r\sqrt{1+r'^2}+r'\sqrt{1+r^2},\ell'+(\sqrt{1+r'^2}-r')\ell\\
+(\sqrt{1+r'^2}+r')m\Big).
\end{multline*}
Moreover, $\caS(\gM)$ is a Fr\'echet nuclear space endowed with the seminorms $(\norm f\norm_{\alpha})$.
\medskip
For $f,h\in\caS(\gM)$, the product $f\star_{\theta,\caP}h$ is well-defined by \eqref{eq-prod}. However, it is not possible to show that it belongs to $\caS(\gM)$ unless we consider this expression as an oscillatory integral. Let us define this concept. For $F\in\caS(\gM^2)$, one can show using integrations by parts that:
\begin{multline}
\int\dd a_i\dd\ell_i\ e^{\frac{2i}{\theta}(\sinh(2a_2)\ell_1-\sinh(2a_1)\ell_2)} F(a_1,a_2,\ell_1,\ell_2)=\\
\int\dd a_i\dd\ell_i\ e^{\frac{2i}{\theta}(\sinh(2a_2)\ell_1-\sinh(2a_1)\ell_2)} \Big(\frac{1-\frac{\theta^2}{4}\partial_{\ell_2}^2}{1+\sinh^2(2a_1)}\Big)^{k_1}\Big(\frac{1-\frac{\theta^2}{4}\partial_{\ell_1}^2}{1+\sinh^2(2a_2)}\Big)^{k_2}\\
\Big(\frac{1-\frac{\theta^2}{16\cosh^2(2a_2)}\partial_{a_2}^2}{1+\ell_1^2}\Big)^{p_1} \Big(\frac{1-\frac{\theta^2}{16\cosh^2(2a_1)}\partial_{a_1}^2}{1+\ell_2^2}\Big)^{p_2}F(a_1,a_2,\ell_1,\ell_2)\\
= \int\dd a_i\dd\ell_i\ e^{\frac{2i}{\theta}(\sinh(2a_2)\ell_1-\sinh(2a_1)\ell_2)} \frac{1}{(1+\sinh^2(2a_1))^{k_1}}\\
\frac{DF(a_1,a_2,\ell_1,\ell_2)}{(1+\sinh^2(2a_2))^{k_2}(1+\ell_1^2)^{p_1}(1+\ell_2^2)^{p_2}}\label{eq-osc}
\end{multline}
for any $k_i,p_i\in\gN$, and where $D$ is a linear combination of products of bounded functions (with every derivatives bounded) in $(a_i,\ell_i)$ with powers of $\partial_{\ell_i}$ and $\frac{1}{\cosh(2a_i)}\partial_{a_i}$. The first expression of \eqref{eq-osc} is not defined for non-integrable functions $F$ bounded by polynoms in $r_i:=\sinh(2a_i)$ and $\ell_i$. However, the last expression of \eqref{eq-osc} is well-defined for $k_i,p_i$ sufficiently large. Therefore it gives a sense to the first expression, now understood as an {\defin oscillatory integral}, i.e. as being equal to the last expression. This definition of oscillatory integral \cite{Bieliavsky:2001os,Bieliavsky:2010kg} is unique, in particular unambiguous in the powers $k_i,p_i$ because of the density of $\caS(\gM)$ in polynomial functions in $(r,\ell)$ of a given degree. Note that this corresponds to the usual oscillatory integral \cite{Hormander:1979} in the coordinates $(r,\ell)$.

The first part of the next Theorem shows that this concept of oscillatory integral is necessary \cite{Bieliavsky:2010kg} for $\caS(\gM)$ to obtain an associative algebra, while the other parts have been treated in \cite{Bieliavsky:2008or}.
\begin{theorem}
Let $\caP:\gR\to C^\infty(\gR)$ be a smooth map such that $\caP_0\equiv1$, and $\caP_\theta(a)$ as well as its inverse are bounded by $C\sinh(2a)^k$, $k\in\gN$, $C>0$.
\begin{itemize}
\item Then, the expression \eqref{eq-prod}, understood as an oscillatory integral, yields a $G$-invariant non-formal deformation quantization.\\
In particular, $(\caS(\gM),\star_{\theta,\caP})$ is a Fr\'echet algebra.
\item For $f,h\in\caS(\gM)$, the map $\theta\mapsto f\star_{\theta,\caP} h$ is smooth and admits a $G$-invariant formal star-product as asymptotic expansion in $\theta=0$.
\item Every $G$-invariant formal star-product on $\gM$ can be obtained as an expansion of a $\star_{\theta,\caP}$, for a certain $\caP$.
\item For $\caP_\theta(a)=\caP_\theta(0)\sqrt{\cosh(a\theta/2)}$, one has the tracial identity: $\int f\star_{\theta,\caP}h=\int f\fois h$.
\end{itemize}
\end{theorem}
We denote $\star_\theta$ the product $\star_{\theta,\caP}$ with $\caP_\theta(a)=\sqrt{\cosh(a\theta/2)}$, therefore satisfying the tracial identity.

\subsection{Schwartz multipliers}

Let us consider the topological dual $\caS'(\gM)$ of $\caS(\gM)$. In the coordinates $(r,\ell)$, it corresponds to tempered distributions. By denoting $\langle-,-\rangle$ the duality bracket between $\caS'(\gM)$ and $\caS(\gM)$, one can extend the product $\star_\theta$ (with tracial identity) as $\forall T\in\caS'(\gM)$, $\forall f,h\in\caS(\gM)$,
\begin{equation*}
\langle T\star_\theta f,h\rangle:=\langle T,f\star_\theta h\rangle\ \text{ and }\ \langle f\star_\theta T,h\rangle:=\langle T,h\star_\theta f\rangle,
\end{equation*}
which is compatible with the case $T\in\caS(\gM)$. Then, we define \cite{Bieliavsky:2013sk}:
\begin{multline*}
\caM_{\star_\theta}(\gM):=\{T\in\caS'(\gM),\ f\mapsto T\star_\theta f,\ f\mapsto f\star_\theta T\text{ are continuous}\\
\text{from }\caS(\gM) \text{ into itself}\},
\end{multline*}
and the product can be extended to $\caM_{\star_\theta}(\gM)$ by:
\begin{equation*}
\forall S,T\in\caM_{\star_\theta}(\gM),\ \forall f\in\caS(\gM)\quad:\quad \langle S\star_\theta T,f\rangle:=\langle S,T\star_\theta f\rangle=\langle T,f\star_\theta S\rangle.
\end{equation*}
We can equipy $\caM_{\star_\theta}(\gM)$ with the topology associated to the seminorms:
\begin{equation*}
\norm T\norm_{B,\alpha,L}=\sup_{f\in B}\norm T\star_\theta f\norm_\alpha\ \text{ and }\ \norm T\norm_{B,\alpha,R}=\sup_{f\in B}\norm f\star_\theta T\norm_\alpha
\end{equation*}
where $B$ is a bounded subset of $\caS(\gM)$, $\alpha\in\gN^4$ and $\norm f\norm_\alpha$ is the Schwartz seminorm introduced in Definition \ref{def-schwartz}. Note that $B$ can be described as a set satisfying $\forall\alpha$, $\sup_{f\in B}\norm f\norm_\alpha$ exists.
\begin{proposition}
$(\caM_{\star_\theta}(\gM),\star_\theta)$ is an associative Hausdorff locally convex complete and nuclear algebra, with separately continuous product, called the {\defin multiplier algebra}.
\end{proposition}

\section{Construction of the star-exponential}

\subsection{Formal construction}

Let us follow the method developed in \cite{Bieliavsky:2013sk}. We want first to find a solution to the following equation
\begin{equation}
\partial_t f_t(a,\ell)=\frac{i}{\theta}(\lambda_X\star_\theta^0 f_t)(a,\ell)\label{eq-defexp}
\end{equation}
for $X=\alpha H+\beta E+\gamma F\in\kg$, with initial condition $\lim_{t\to 0}f_t(a,\ell)=1$. To remove the integral of this equation, we apply the partial Fourier transformation \eqref{eq-fourier} to obtain
\begin{align*}
\caF(\lambda_{H}\star_\theta^0 f)= \left(2i\partial_\xi+\frac{i\theta}{2}\partial_a)\right)&\hat f,\qquad 
\caF(\lambda_{E}\star_\theta^0 f)=e^{-2a-\frac{\theta\xi}{2}}\hat f,\\
\caF(\lambda_{F}&\star_\theta^0 f)=-e^{2a+\frac{\theta\xi}{2}}\hat f
\end{align*}
so that the equation \eqref{eq-defexp} can be reformuled as
\begin{equation*}
\partial_t \hat f_t(a,\xi)=\frac{i}{\theta}\Big[2i\alpha\partial_\xi+\frac{i\theta\alpha}{2}\partial_a+\beta e^{-2a-\frac{\theta\xi}{2}}-\gamma e^{2a+\frac{\theta\xi}{2}}\Big]\hat f_t(a,\xi).
\end{equation*}
The existence of a solution of this equation which satisfies the BCH property directly relies on the covariance \eqref{eq-cov} of the Moyal product. We have the explicit following result.
\begin{proposition}
For $X=\alpha H+\beta E+\gamma F\in\kg$, the expression
\begin{equation*}
E_{\star_\theta^0}(t\lambda_X)(a,\ell)=e^{\frac{i}{\theta}\Big(2\ell\alpha t+\frac{1}{\alpha}\sinh(\alpha t)(\beta e^{-2a}-\gamma e^{2a})\Big)}
\end{equation*}
is a solution of the equation \eqref{eq-defexp} with initial condition $\lim_{t\to 0}f_t(a,\ell)=1$. Moreover, it satisfies the BCH property: $\forall X,Y\in\kg$,
\begin{equation}
E_{\star_\theta^0}(\lambda_{\text{BCH}(X,Y)})=E_{\star_\theta^0}(\lambda_X)\star_\theta^0 E_{\star_\theta^0}(\lambda_Y).\label{eq-bchStar}
\end{equation}
\end{proposition}
\begin{proof}
By performing the following change of variables $b=-\frac{a}{\alpha}-\frac{\theta\xi}{4\alpha}$ and $c=-\frac{a}{\alpha}+\frac{\theta\xi}{4\alpha}$, we reformulate the equation as
\begin{equation*}
\partial_t \hat f_t=\partial_b\hat f_t+\frac{i\beta}{\theta} e^{2\alpha b}-\frac{i\gamma}{\theta}e^{-2\alpha b},
\end{equation*}
whose solution is given by $\hat f_t(b,c)=\exp\Big(-\int^b_0(\frac{i\beta}{\theta} e^{2\alpha s}-\frac{i\gamma}{\theta}e^{-2\alpha s})\dd s\Big)h(b+t,c)$, where $h$ is an arbitrary function. By assuming the initial condition, we obtain the expression of $E_{\star_\theta^0}(tX)(a,\ell)$. The BCH property is given by direct computations from the expression of the product \eqref{eq-moyal} and from \eqref{eq-bchGroup}.
\end{proof}

Finally, we push this solution by $T_\theta$:
\begin{multline*}
E_{\star_{\theta,\caP}}(tX)(a,\ell):=T_\theta E_{\star_\theta^0}(tT_\theta^{-1}\lambda_X)(a,\ell)\\
=\frac{\caP_\theta(0)\cosh(\alpha t)}{\caP_\theta(\frac{2\alpha t}{\theta})}e^{\frac{i}{\theta}\sinh(\alpha t)(2\ell+\frac{\beta}{\alpha}e^{-2a}-\frac{\gamma}{\alpha} e^{2a})+\frac{2\caP_\theta'(0)}{\theta\caP_\theta(0)}\alpha t}.
\end{multline*}
It also satisfies the BCH property \eqref{eq-bchStar}.

\subsection{Multiplier property}

For the star-product with tracial property, we want to define the star-exponential at the non-formal level. We can use the oscillatory integral in the star-product to show \cite{Bieliavsky:2013sk}:
\begin{theorem}
For any $X\in\kg$, the function
\begin{equation*}
E_{\star_{\theta}}(tX)(a,\ell)=\sqrt{\cosh(\alpha t)}e^{\frac{i}{\theta}\sinh(\alpha t)(2\ell+\frac{\beta}{\alpha}e^{-2a}-\frac{\gamma}{\alpha} e^{2a})}
\end{equation*}
lies in the multiplier algebra $\caM_{\star_\theta}(\gM)$.
\end{theorem}

As it belongs to a specific ``functional space'', the function $E_{\star_\theta}$ is called the {\defin non-formal star-exponential} of the group $G$ for the star-product $\star_\theta$. The BCH property
\begin{equation*}
E_{\star_\theta}(\text{BCH}(X,Y))=E_{\star_\theta}(X)\star_\theta E_{\star_\theta}(Y)
\end{equation*}
now makes sense in the topological space $\caM_{\star_\theta}(\gM)$. This functional framework is useful for applications of the star-exponential discussed in the introduction.

\bibliographystyle{utcaps}
\bibliography{biblio-these,biblio-perso,biblio-recents}

\end{document}